\documentclass[12pt]{amsart}
\usepackage{amsmath,amssymb,amsbsy,amsfonts,latexsym,amsopn,amstext,
                                               amsxtra,euscript,amscd}

\newfont{\teneufm}{eufm10}
\newfont{\seveneufm}{eufm7}
\newfont{\fiveeufm}{eufm5}
\newfam\eufmfam
                \textfont\eufmfam=\teneufm \scriptfont\eufmfam=\seveneufm
                \scriptscriptfont\eufmfam=\fiveeufm
\def\bbbc{{\mathchoice {\setbox0=\hbox{$\displaystyle\rm C$}\hbox{\hbox
to0pt{\kern0.4\wd0\vrule height0.9\ht0\hss}\box0}}
{\setbox0=\hbox{$\textstyle\rm C$}\hbox{\hbox
to0pt{\kern0.4\wd0\vrule height0.9\ht0\hss}\box0}}
{\setbox0=\hbox{$\scriptstyle\rm C$}\hbox{\hbox
to0pt{\kern0.4\wd0\vrule height0.9\ht0\hss}\box0}}
{\setbox0=\hbox{$\scriptscriptstyle\rm C$}\hbox{\hbox
to0pt{\kern0.4\wd0\vrule height0.9\ht0\hss}\box0}}}}
\def\bbbq{{\mathchoice {\setbox0=\hbox{$\displaystyle\rm
Q$}\hbox{\raise 0.15\ht0\hbox to0pt{\kern0.4\wd0\vrule
height0.8\ht0\hss}\box0}} {\setbox0=\hbox{$\textstyle\rm
Q$}\hbox{\raise 0.15\ht0\hbox to0pt{\kern0.4\wd0\vrule
height0.8\ht0\hss}\box0}} {\setbox0=\hbox{$\scriptstyle\rm
Q$}\hbox{\raise 0.15\ht0\hbox to0pt{\kern0.4\wd0\vrule
height0.7\ht0\hss}\box0}} {\setbox0=\hbox{$\scriptscriptstyle\rm
Q$}\hbox{\raise 0.15\ht0\hbox to0pt{\kern0.4\wd0\vrule
height0.7\ht0\hss}\box0}}}}
\def\bbbt{{\mathchoice {\setbox0=\hbox{$\displaystyle\rm
T$}\hbox{\hbox to0pt{\kern0.3\wd0\vrule height0.9\ht0\hss}\box0}}
{\setbox0=\hbox{$\textstyle\rm T$}\hbox{\hbox
to0pt{\kern0.3\wd0\vrule height0.9\ht0\hss}\box0}}
{\setbox0=\hbox{$\scriptstyle\rm T$}\hbox{\hbox
to0pt{\kern0.3\wd0\vrule height0.9\ht0\hss}\box0}}
{\setbox0=\hbox{$\scriptscriptstyle\rm T$}\hbox{\hbox
to0pt{\kern0.3\wd0\vrule height0.9\ht0\hss}\box0}}}}
\def\bbbs{{\mathchoice
{\setbox0=\hbox{$\displaystyle     \rm S$}\hbox{\raise0.5\ht0\hbox
to0pt{\kern0.35\wd0\vrule height0.45\ht0\hss}\hbox
to0pt{\kern0.55\wd0\vrule height0.5\ht0\hss}\box0}}
{\setbox0=\hbox{$\textstyle        \rm S$}\hbox{\raise0.5\ht0\hbox
to0pt{\kern0.35\wd0\vrule height0.45\ht0\hss}\hbox
to0pt{\kern0.55\wd0\vrule height0.5\ht0\hss}\box0}}
{\setbox0=\hbox{$\scriptstyle      \rm S$}\hbox{\raise0.5\ht0\hbox
to0pt{\kern0.35\wd0\vrule height0.45\ht0\hss}\raise0.05\ht0\hbox
to0pt{\kern0.5\wd0\vrule height0.45\ht0\hss}\box0}}
{\setbox0=\hbox{$\scriptscriptstyle\rm S$}\hbox{\raise0.5\ht0\hbox
to0pt{\kern0.4\wd0\vrule height0.45\ht0\hss}\raise0.05\ht0\hbox
to0pt{\kern0.55\wd0\vrule height0.45\ht0\hss}\box0}}}}
\def\bbbz{{\mathchoice {\hbox{$\sf\textstyle Z\kern-0.4em Z$}}
{\hbox{$\sf\textstyle Z\kern-0.4em Z$}} {\hbox{$\sf\scriptstyle
Z\kern-0.3em Z$}} {\hbox{$\sf\scriptscriptstyle Z\kern-0.2em
Z$}}}}

\newtheorem{theorem}{Theorem}
\newtheorem{lemma}[theorem]{Lemma}

\def\squareforqed{\hbox{\rlap{$\sqcap$}$\sqcup$}}
\def\qed{\ifmmode\squareforqed\else{\unskip\nobreak\hfil
\penalty50\hskip1em\null\nobreak\hfil\squareforqed
\parfillskip=0pt\finalhyphendemerits=0\endgraf}\fi}

\def\cD{{\mathcal D}}

\def\cI{{\mathcal I}}

\def\cL{{\mathcal L}}

\def\cS{{\mathcal S}}

\def\cU{{\mathcal U}}

\def\cW{{\mathcal W}}

\def \sf {\mathfrak s}

\def\bS{\overline{\mathcal S}}

\def\bI{\overline{\mathcal I}}

\newcommand{\ignore}[1]{}

\def \F{\mathbb{F}}

\def\mand{\qquad\mbox{and}\qquad}

\def\\{\cr}
\def\({\left(}
\def\){\right)}
\def\fl#1{\left\lfloor#1\right\rfloor}
\def\rf#1{\left\lceil#1\right\rceil}

\begin{document}


\title[Character Sums and  Root Finding]{Character Sums and
Deterministic Polynomial Root Finding in Finite Fields}

\author{Jean~Bourgain}
\address{Institute for Advanced Study,
Princeton, NJ 08540, USA}
\email{bourgain@ias.edu}

\author{Sergei V.~Konyagin}
\address{Steklov Mathematical Institute,
8, Gubkin Street, Moscow, 119991, Russia}
\email{konyagin@mi.ras.ru}

\author{Igor E. Shparlinski}
\address{Department of Pure Mathematics, University of New South Wales,
Sydney, NSW 2052, Australia}
\email{igor.shparlinski@unsw.edu.au}

\begin{abstract}  We obtain a new bound of certain double multiplicative character
sums. We use this bound together with some other previously obtained
results to obtain new algorithms for finding roots of
polynomials modulo a prime $p$.
\end{abstract}


\maketitle

\section{Introduction}

Let $\F_q$ be a finite field of $q$ elements of characteristic $p$.
The classical algorithm of Berlekamp~\cite{Ber} reduces the
problem of factoring polynomials of degree $n$ over $\F_q$ to the
problem of factoring squarefree polynomials of degree $n$ over $\F_p$
that fully split in $\F_p$, see also~\cite[Chapter~14]{vzGG}.
Shoup~\cite[Theorem~3.1]{Shoup} has given  a
deterministic algorithm  that fully factors any polynomial
of degree $n$ over $\F_p$ in $O(n^{2 + o(1)} p^{1/2} (\log p)^2)$
arithmetic operations over $\F_p$; in particular it runs in time
$n^{2} p^{1/2+o(1)}$. Furthermore,
Shoup~\cite[Remark~3.5]{Shoup} has also announced
an algorithm of complexity $O(n^{3/2 + o(1)} p^{1/2} (\log p)^2)$
for factoring arbitrary univariate polynomials of degree $n$ over $\F_p$.

We remark, that although the efficiency of deterministic polynomial
factorisation algorithms falls far behind
the fastest probabilistic algorithms, see,
for example,~\cite{vzGSh,KaSh,KeUm}, the question is
of great theoretic interest.

Here we address a special case of the polynomial factorisation problem when
the polynomial $f$ fully splits over $\F_p$ (as we have noticed
there is a polynomial time reduction between factoring general
polynomials and polynomials that split over $\F_p$). That is,
here we deal with the root finding problem. We also note that
in order to find a root (or all roots) of a polynomial $f \in \F_p[X]$,
it is enough to do the same for the polynomial $\gcd\(f(X), X^{p-1}-1\)$
which is squarefree fully splits over $\F_p$.

We consider two variants of the root finding problem:
\begin{itemize}
\item Given  a polynomial $f \in \F_p[X]$, find all roots of $f$ in $\F_p$.
\item Given  a polynomial $f \in \F_p[X]$, find at least one root of $f$ in $\F_p$.
\end{itemize}

For the case of finding all roots we show that essentially the initial
approach of Shoup~\cite{Shoup} together with the fast factor refinement procedure
of Bernstein~\cite{Bern} lead to an algorithm of complexity $n p^{1/2 +o(1)}$.
In fact this result is already implicit in~\cite{Shoup} but here we record
it again  with a very short proof. We use this as a benchmark for
our algorithm for the second problem.

We remark that a natural example of the situation when one has to
find a root of a polynomial of large degree arises in the problem of
constructing elliptic curves over $\F_p$  with prescribed number
of $\F_p$-rational points. In this case one has to find a root of the
{\it Hilbert class polynomial\/}, we refer to~\cite{Suth1,Suth2} for more
detail on this and underlying problems.

In the case of finding just one root, we obtain a faster algorithm,
which is based on bounds of double multiplicative character sums
$$
T_\chi(\cI, \cS) = \sum_{u \in \cI} \left|  \sum_{s\in \cS}    \chi(u+s)\right|^2,
$$
where $\cI = \{1, \ldots, h\}$ is an interval of $h$ consecutive
integers, $\cS \subseteq \F_p$ is  an arbitrary  set and $\chi$ is a
multiplicative character of $\F_p^*$.
More precisely, here we use a new  bound on $T_\chi(\cI, \cS)$
to improve the bound $np^{1/2+o(1)}$ in
the case when $n$ is large enough, namely if it grows as a power of $p$.
We believe that our new bound of the sums $T_\chi(\cI, \cS)$ as well as
several auxiliary results (based on some methods  from additive
combinatorics) are of independent
interest as well.

Throughout the paper, any implied constants in symbols $O$ and $\ll$
 may depend on two real positive parameters $\varepsilon$ and  $\delta$
and  are absolute otherwise. We recall
that the notations $U = O(V)$ and  $U \ll V$ are
all equivalent to the statement that $|U| \le c V$ holds
with some constant $c> 0$. We also use $U \asymp V$ to denote
that $U \ll V \ll U$.

\section{Bounds on The Number Solutions to Some Equations
and Character Sums}

\subsection{Uniform distribution and exponential sums}

The following result is well-known and can be found, for example, in~\cite[Chapter~1, Theorem~1]{Mont}
(which is a  more precise form of the celebrated Erd{\H o}s--Tur\'{a}n inequality).

\begin{lemma}
\label{lem:ET small int}
Let $\xi_1, \ldots, \xi_M$ be a sequence of $M$ points of the unit interval $[0,1]$.
Then for any integer $K\ge 1$, and an interval $[0,\rho] \subseteq [0,1]$,
we have
\begin{equation*}
\begin{split}
\# \{m =1, \ldots, M~:&~\xi_m  \in [0,\rho]\} - \rho M\\
\ll \frac{M}{K} + &\sum_{k=1}^K \(\frac{1}{K} +\min\{\rho, 1/k\}\)
\left|\sum_{m=1}^M \exp(2 \pi i k \xi_m)\right|.
\end{split}
\end{equation*}
\end{lemma}

\subsection{Preliminary bounds}
\label{sec:prelim}

Throughout this section we
fix some set $\cS\subseteq \F_p$ of
and interval $\cI = \{1, \ldots, h\}$  of  $h \le  p^{1/2}$ consecutive
integers.

We say that a set $\cD \subseteq \F_p$ is $\Delta$-spaced
if no elements $d_1, d_2\in \cD$ and positive integer $k\le\Delta$
satisfy the equality $d_1+k=d_2$.

Here we always assume that the set $\cS$ is $h$-spaced.

Finally, we also fix some $L$ and
denote by $\cL$  the set of primes of the interval $[L, 2L]$.

We denote
\begin{equation*}
\begin{split}
\cW  =\Bigl\{(u_1,u_2,\ell_1,\ell_2,s_1,s_2)\in \cI^2&\times \cL^2 \times \cS^2~:\\
&~
\frac{u_1+s_1}{\ell_1} \equiv \frac{u_2+s_2}{\ell_2} \pmod p\Bigr\}.
\end{split}
\end{equation*}

The following result is based on some ideas of Shao~\cite{Shao}.

\begin{lemma}
\label{lem:W}  If $L < h$ and $2hL < p$ then
$$
\# \cW\ll (\# \cS h L)^2p^{-1}  +  \# \cS h L p^{o(1)}.
$$
\end{lemma}

\begin{proof} Clearly
\begin{equation}
\label{eq:bound1}
\# \cW = \# \cW^* + O(\# \cS h L) ,
\end{equation}
where
$$
\cW^* = \{(u_1,u_2,\ell_1,\ell_2,s_1,s_2)\in \cW~:~\ell_1 \ne \ell_2\}.
$$

Denote
$$
\bS = \cS + \cI = \{u+s~:~(u,v) \in \cI \times \cS\},\quad \bI = \{-h, \ldots, h\}.
$$
Clearly
\begin{equation*}
\begin{split}
\cW^* \ll h^{-2} \Bigl\{(u_1,u_2,\ell_1,\ell_2,s_1,s_2)\in \bI^2&\times \cL^2 \times \bS^2~:~
\ell_1 \ne \ell_2,\\
& \frac{u_1+s_1}{\ell_1} \equiv \frac{u_2+s_2}{\ell_2} \pmod p\Bigr\}.
\end{split}
\end{equation*}

Note that for fixed $\ell_1,\ell_2 \in\cL$, $\ell_1 \ne \ell_2$ and
integer $x$, $|x|\le2h L$ the congruence
$$u_1 \ell_2 - u_2\ell_1  \equiv x \pmod p
$$
is equivalent to the equation $u_1 \ell_2 - u_1 \ell_2 = x$
(since $ 2h L< p$) and thus has $O\(h/L\)$ solutions.
We rewrite
$$
\frac{u_1+s_1}{\ell_1} \equiv \frac{u_2+s_2}{\ell_2} \pmod p
$$
as
$$
s_1 \ell_2 - s_2 \ell_1 \equiv x \equiv u_1 \ell_2 - u_2\ell_1  \pmod p.
$$
One can consider that $x\ge0$. We now bound the cardinality of
\begin{equation*}
\begin{split}
\cU =\Bigl\{(x,\ell_1,\ell_2,s_1,s_2)\in [0,  2hL]&\times \cL^2 \times \bS^2~:
\\
& s_1 \ell_2 - s_2 \ell_1 \equiv x \pmod p\Bigr\}.
\end{split}
\end{equation*}
The above argument shows that
\begin{equation}
\label{eq:bound2}
\begin{split}
\cW^* \le h^{-2 } (h/L) \# \cU = h^{-1} L^{-1}  \# \cU.
\end{split}
\end{equation}

We now apply  Lemma~\ref{lem:ET small int} to the sequence of fractional parts
$$
\left\{\frac{s_1 \ell_2 - s_2 \ell_1}{p}\right\},
\qquad (\ell_1,\ell_2,s_1,s_2)\in  \cL^2 \times \bS^2,
$$
with $M = (\# \cL)^2  (\# \bS)^2$, $\rho = 2hL p^{-1}$
and $K = \rf{\rho^{-1}}$. This yields the bound
\begin{equation*}
\begin{split}
\# \cU \ll  (\# \cL)^2 & (\# \bS)^2 \rho \\
& + \rho
\sum_{k=1}^K
\left|\sum_{(\ell_1,\ell_2,s_1,s_2)\in  \cL^2 \times \bS^2}
\exp\(2 \pi i k \(s_1 \ell_2 - s_2 \ell_1\)/p\)\right|\\
= (\# \cL)^2 & (\# \bS)^2 \rho  + \rho
\sum_{k=1}^K
\left|\sum_{(\ell,s)\in  \cL \times \bS}
\exp\(2 \pi i k s \ell/p\)\right|^2.
\end{split}
\end{equation*}
Using the Cauchy inequality, denoting $r = k \ell$ and then using the classical
bound on the divisor function, we derive
\begin{equation*}
\begin{split}
\# \cU \ll   (\# \cL)^2 & (\# \bS)^2 \rho  + \rho \# \cL
\sum_{k=1}^K \sum_{\ell \in \cL}
\left|\sum_{s\in  \bS}
\exp\(2 \pi i k s \ell/p\)\right|^2\\
\ll   (\# \cL)^2 & (\# \bS)^2 \rho  + p^{o(1)} \rho \# \cL \sum_{r=0}^{p-1}
\left|\sum_{s\in  \bS}
\exp\(2 \pi ir s/p\)\right|^2,
\end{split}
\end{equation*}
since $r \in [1, 2KL] \subseteq [0, p-1]$ provided that
$p$ is sufficiently large. Thus, using the Parseval inequality
and recalling the values of our parameters, we obtain
$$
\# \cU  \ll h L^3  (\# \bS)^2 p^{-1}
+  hL^2 \# \bS p^{o(1)} .
$$
Using the trivial bound $\# \bS \ll \# \cS h$, we obtain
$$
\# \cU  \ll h^3 L^3 (\# \cS)^2  p^{-1}
+  h^2L^2 \# \cS p^{o(1)}.
$$
Thus, recalling~\eqref{eq:bound1} and~\eqref{eq:bound2} we conclude the proof.
\end{proof}

Denote
\begin{equation}
\begin{split}
\label{eq:Wxy}
W(x&,y)  \\
& =   \#\left\{(u,\ell,s,t)\in \cI \times \cL \times \cS^2~:~\frac{u+s}{\ell} = x,\ \frac{u+t}{\ell} = y\right\}.
\end{split}
\end{equation}

\begin{lemma}
\label{lem:W2}
 We have
$$
\sum_{x,y \in \F_p} W(x,y)^2 \ll
(\# S)^3 (h L)^2p^{-1}  +  (\# S)^2 h L p^{o(1)}.
$$
\end{lemma}

\begin{proof} Clearly
\begin{equation*}
\begin{split}
\sum_{x,y \in \F_p} W(x&,y)^2\\
 & =   \#\Bigr\{(u_1,u_2,\ell_1,\ell_2,s_1,t_1,s_2,t_2)\in \cI^2 \times \cL^2 \times \cS^4~:\\
 & \qquad\qquad\qquad \frac{u_1+s_1}{\ell_1} = \frac{u_2+s_2}{\ell_2}, \
 \frac{u_1+t_1}{\ell_1} = \frac{u_2+t_2}{\ell_2}\Bigl\}.
\end{split}
\end{equation*}
For each $(u_1,u_2,\ell_1,\ell_2,s_1,s_2) \in \cW$ and $t_1 \in \cS$ there is
only one possible values for $t_2$. The result now follows from Lemma~\ref{lem:W}.
 \end{proof}

\subsection{Character sum estimates}

First we recall the following special case of the Weil bound of character sums
(see~\cite[Theorem~11.23]{IwKow}).

\begin{lemma}
\label{lem:Weil}
For any polynomial $F(X) \in \F_p[X]$ with  $N$ distinct zeros
in  the algebraic closure of $\F_p$ and
which is not a perfect $d$th power in the ring of polynomials
over $\F_p$,
 and  a nonprincipal multiplicative character $\chi$ of $\F_p^*$
 of order $d$, we have
$$
\left| \sum_{x \in \F_p}\chi\(F(x)\)\right| \le N p^{1/2}.
$$
\end{lemma}

The following estimate  improves and 
generalises~\cite[Lemma~14]{BGKS} and also~\cite[Theorem~8]{Chang}.
It proof  is 
based on the classical ``amplification'' 
argument of Burgess~\cite{Burg1,Burg2}.

\begin{lemma}
\label{lem:DoubleSums}
For any  positive $\delta > 0$ there is some $\eta > 0$ such that for
an  interval $\cI = \{1, \ldots, h\}$  of  $h \le  p^{1/2}$ consecutive
integers
and any  $h$-spaced set  $\cS \subseteq \F_p$
with
$$\# \cS h >  p^{1/2 + \delta},
$$
for any nontrivial multiplicative character $\chi$ of $\F_p^*$ we have
$$
 T_\chi(\cI, \cS) \ll  (\# S)^2 h  p^{ - \eta}.
$$
\end{lemma}

\begin{proof} We choose a sufficiently small $\varepsilon$ and
define
$$
L = \fl{hp^{- 2\varepsilon}} \mand T = \fl{p^{\varepsilon}}.
$$
As in Section~\ref{sec:prelim},  we
denote by $\cL$  the set of primes of the interval $[L, 2L]$.
Note that
$$
\(\#\cS\)^2T L \ll  (\# S)^2 hp^{  - \varepsilon}.
$$
Then
\begin{equation}
\label{eq:T-Sigma}
\begin{split}
T_\chi(\cI, \cS)  & = \frac{1}{(T+1)\#\cL}\sigma  + O\(\(\#\cS\)^2T L\)\\
& = \frac{1}{(T+1)\#\cL}\sigma  + O\( (\# S)^2 hp^{  - \varepsilon}\),
\end{split}
\end{equation}
where
\begin{equation*}
\begin{split}
\sigma_   = & \sum_{\ell \in \cL}
\sum_{t = 0}^T   \sum_{u \in  \widetilde  \cI}  \sum_{s_1,s_2\in \cS}   \chi(u+s_1 + t \ell) \overline \chi(u+s_2 + t\ell) \\
    = & \sum_{u \in  \widetilde  \cI} \sum_{\ell \in \cL}   
        \sum_{s_1,s_2\in \cS}  \sum_{t = 0}^T  \chi\(\frac{u+s_1}{\ell} + t\)
\overline \chi\(\frac{u+s_2}{\ell} + t\).
\end{split}
\end{equation*}
Furthermore,
$$
\sigma   = \sum_{x,y \in \F_p} W(x,y) \sum_{t = 0}^T  \chi\(x + t\)
\overline \chi\(y + t\),
$$
where $W(x,y)$ is defined by~\eqref{eq:Wxy}.

Therefore, for any integer $\nu \ge 1$ by the H{\"o}lder inequality,  we have
\begin{equation}
\label{eq:Sigma}
\begin{split}
\sigma^{2\nu} \le \sum_{x,y \in \F_p}  W(x,y)^2 &  \(\sum_{x,y \in \F_p} W(x,y)\)^{2\nu -2}\\
 & \sum_{x,y \in \F_p}\left| \sum_{t = 0}^T  \chi\(x + t\) \overline \chi\(y + t\) \right|^{2 \nu}.
 \end{split}
\end{equation}
Clearly
\begin{equation}
\label{eq:Sum-W1}
\sum_{x,y \in \F_p} W(x,y) \ll \# \cI \# \cL \(\# \cS\)^2   \ll  (\# S)^2 hL.
\end{equation}
We also have
\begin{equation*}
\begin{split}
 \sum_{x,y \in \F_p}\left| \sum_{t = 0}^T  \chi\(x + t\) \overline \chi\(y + t\) \right|^{2 \nu}&\\
 =  \sum_{t_1, \ldots, t_{2\nu} = 0}^T &
\left | \sum_{x\in \F_p}\prod_{i=1}^\nu \chi\(x + t_i\) \prod_{i=\nu+1}^{2\nu}\overline \chi\(x + t_i\) \right|^2.
  \end{split}
\end{equation*}
Using the Weil bound in the form of Lemma~\ref{lem:Weil} if
$(t_1, \ldots, t_\nu)$ is not a permutation of $(t_{\nu +1}, \ldots, t_{2\nu})$, 
and the trivial bound otherwise, we derive
$$
\sum_{x,y \in \F_p}\left| \sum_{t = 0}^T  \chi\(x + t\)
\overline \chi\(y + t\) \right|^{2 \nu}
 \ll T^{2\nu} p +  T^{\nu} p^2
$$
(see also~\cite[Lemma~12.8]{IwKow} that underlies the Burgess method).
Taking $\nu$ to be large enough so that
$T^{2\nu} p > T^{\nu} p^2$ we obtain
\begin{equation}
\label{eq:Sum-S-2nu}
\sum_{x,y \in \F_p}\left| \sum_{t = 0}^T  \chi\(x + t\) \overline \chi\(y + t\) \right|^{2 \nu}
 \ll T^{2\nu} p .
\end{equation}

Substituting~\eqref{eq:Sum-W1} and~\eqref{eq:Sum-S-2nu} in~\eqref{eq:Sigma} we obtain
$$
\sigma^{2\nu} \ll  T^{2\nu} p \( (\# S)^2 hL\)^{2\nu-2}
\sum_{x,y \in \F_p} W(x,y)^2 .
$$
We now apply Lemma~\ref{lem:W2} to derive
\begin{equation}
\begin{split}
\label{eq:SigmaW}
\sigma^{2\nu} &\ll  T^{2\nu} p \( (\# S)^2 hL\)^{2\nu-2}
\( (\# S)^3 (hL)^2 p^{-1}  +   (\# S)^2 hL p^{o(1)}\)\\
 &\ll  T^{2\nu} p^{1+o(1)} \( (\# S)^2 hL\)^{2\nu}
\((\#\cS)^{-1} p^{-1}  + (\# S)^{-2} h^{-1}L^{-1}\) .
\end{split}
\end{equation}
Taking a sufficiently small $\varepsilon>0$, we obtain
$$
 (\# S)^2 h L > p^{1 +\delta}
$$
which together with~\eqref{eq:T-Sigma} concludes the proof.
\end{proof}

\section{Root Finding Algorithms}

\subsection{Finding all roots}

Here we address the question of finding all roots of
a polynomial $f \in \F_p[X]$.

We refer to~\cite{vzGG}
for description of efficient (in particular, polynomial
time) algorithms of polynomial arithmetic
over finite fields such as multiplication, division with 
remainder and computing  the greatest common divisor. 

\begin{theorem}
\label{thm:Factor}  There is a deterministic algorithm that,
given a squarefree polynomial $f \in \F_p[X]$ of degree $n$
 that fully splits over $\F_p$,
finds all roots of $f$ in time $np^{1/2+ o(1)}$.
\end{theorem}

\begin{proof} We set
$$
h =  \fl {p^{1/2} (\log p)^2}.
$$
We now compute the polynomials
\begin{equation}
\label{eq:gcd0}
g_u(X) = \gcd\(f(X), (X+u)^{(p-1)/2}-1\), \quad  u =0,\ldots,h.
\end{equation}
We remark that to compute the greatest common divisor in~\eqref{eq:gcd0}
we first use repeated squaring to compute the residue
$$
H_u(X) \equiv   (X+u)^{(p-1)/2} \pmod {f(X)}, \qquad \deg H_u < n
$$
and then compute
$$
g_u(X) = \gcd\(f(X), H_u(X)\).
$$

If $a \in \F_p$ is a root of $f$ then $(X-a)\mid g_u(X)$ if 
and only if $a+u\neq0$ and $a+u$ is a quadratic residue in $\F_p$.

We now note that the Weil bound on incomplete character sums implies  that for any
two roots $a,b \in \F_p$ of $f$ there is $u\in [0,h]$ such that
\begin{equation}
\label{eq:split}
(X-a)\mid g_u(X) \mand (X-b)\nmid g_u(X).
\end{equation}
Note that the argument of~\cite[Theorem~1.1]{Shp} shows that one
can take $h =  \fl {Cp^{1/2}}$ for some absolute constant $C>0$ just getting 
some minor speed up of this and the original algorithm of Shoup~\cite{Shoup}.

We now recall the factor refinement algorithm of Bernstein~\cite{Bern},
that, in particular, for any set of $N$ polynomial $G_1, \ldots, G_N \in \F_p[X]$ of degree $n$ over $\F_p$
in time $O(nNp^{o(1)})$ finds a set of relatively prime polynomials
$H_1, \ldots, H_M \in \F_p[X]$ such that any polynomial $G_i$, $i=1, \ldots, N$, is a
product of powers of the polynomials $H_1, \ldots, H_M$. Applying this algorithm
to  the family of polynomials $g_u$, $u =0,\ldots,h$, and recalling~\eqref{eq:split}, we see
that it outputs the set of polynomials with
$$
\{H_1, \ldots, H_M\} = \{X-a~:~f(a) = 0\},
$$
which concludes the proof.
\end{proof}

\subsection{Finding one root}

Here we give an algorithm that finds one root of a
polynomial over $\F_p$. It is easy to see that up
to a logarithmic factor this problem
is equivalent to a problem of finding any nontrivial
factor of  a polynomial.

\begin{lemma}
\label{lem:one root}  There is a deterministic algorithm that,
given  a squarefree polynomial $f \in \F_p[X]$ of degree $n>1$
 that fully splits over $\F_p$,
finds in time $(n+p^{1/2})p^{o(1)}$ a factor $g \mid f$
of degree $1 \le \deg g < n$.
\end{lemma}

\begin{proof}
It suffices to prove that for any $\delta>0$ there is a desirable algorithm
with running time at most $(n+p^{1/2})p^{\delta+o(1)}$. If $n\le p^\delta$
then the result follows from Theorem~\ref{thm:Factor}. Now assume that
$\delta $ is small and $n> p^\delta$. Let
$$
h =  \fl {(1+ n^{-1}p^{1/2})p^{\delta/2}}.
$$

We start with computing the polynomials
\begin{equation}
\label{eq:gcd1}
\gcd\(f(X), f(X+u)\), \quad  u =1,\ldots,h,
\end{equation}
see~\cite{vzGG} for fast greatest common divisor algorithms. 
Clearly, if $f$ has two distinct roots $a$ and $b$ with
$|a - b| \le h$ then one
of the polynomials~\eqref{eq:gcd1}
gives a nontrivial factor of $f$. It is also easy to
see that the complexity of this step is at most $n h p^{o(1)}$.

If this step does not produce any nontrivial factor
of $f$ then we note that the set  $\cS$ of the roots of $f$
is $h$-spaced.  We now again compute the polynomials
$g_u(X)$, given by~\eqref{eq:gcd0},  for every $u \in \cI$.

So, we see that for the above choice of $h$
the condition of Lemma~\ref{lem:DoubleSums} holds
and implies that
there is $u \in \cI$ with
$$
 \left|  \sum_{s \in \cS}   \(\frac{s +  u}{p}\) \right| \ll \#\cS p^{-\eta}  = np^{-\eta}.
$$
for some $\eta> 0$ that depends only on $\delta$, and thus the
sequence of Legendre symbols $ \((s + u)/p\)$, $s \in \cS$, cannot
be constant.

Therefore, at least one of the polynomials~\eqref{eq:gcd0}
gives a nontrivial factor of $f$.
As in~\cite{Shoup}, we see that the complexity of this algorithm is
again
$O\(n h (\log p)^{O(1)}\)$. 
Since $\delta>0$ is an arbitrary,
we obtain the desired result.
\end{proof}

\begin{theorem}
\label{thm:one root}  There is a deterministic algorithm that,
given  a squarefree polynomial $f \in \F_p[X]$ of degree $n$
 that fully splits over $\F_p$,
finds in time  $(n+p^{1/2})p^{o(1)}$ a root of $f$.
\end{theorem}

\begin{proof} We use Lemma~\ref{lem:one root}
to find a polynomial factor $g_1$ of $f$ with $1 \le \deg g \le 0.5\deg f$.
Next, we find a polynimial factor $g_2$ of $g_1$ with
$1 \le \deg g_2 \le 0.5 \deg g_1$, and so on. The number of iterations
is $O(\log n)$, and the complexity of each iteration, by Lemma~\ref{lem:one root},
does not exceed $(n+p^{1/2})p^{o(1)}$. This completes the proof.
\end{proof}

\section{Comments}

It is certainly natural to expect that  the condition of Lemma~\ref{lem:DoubleSums}
can be relaxed, however proving such a result seems to be presently out of reach 
(even under the standard number theoretic conjectures). 
Furthermore, such an improvement does not immediately propagate 
into improvements of Theorems~\ref{thm:Factor} 
and~\ref{thm:one root}. It seems that within the method of Shoup~\cite{Shoup}
the only plausible way to reduce the complexity below $p^{1/2}$ is 
to obtain nontrivial bounds of single sums of Legendre sympbols 
$$
\left| \sum_{x=1}^H   \(\frac{(x +  s_1)(x+s_2)}{p}\) \right| \le H p^{-\eta}
$$
for intervals of length $H \ge p^{\alpha}$ with some fixed $\alpha < 1/2$, 
uniformly over $s_1,s_2 \in \F_p$, $s_1 \ne s_2$. 
It seems that even the Generalised Riemann Hypothesis (GRH) does not 
immediately imply such a statement. In fact, even in the case of 
linear polynomials, it is not known how to use the GRH 
to get an improvement  of the Burgess bound~\cite{Burg1,Burg2}
(for intervals away from the origin). 

\section*{Acknowledgement}

The authors are grateful to Andrew Sutherland for patient 
explanations of several issues  related to Hilbert class polynomials. 

The research of  J.~B. was partially supported by National Science
Foundation,  Grant DMS-0808042, that of  S.~V.~K. by Russian Fund
for Basic Research Grant N.~14-01-00332, and Program Supporting
Leading Scientific Schools, Grant Nsh-3082.2014.1, and that of
I.~E.~S. by Australian Research Council, Grant 
DP130100237.

\end{document}